\documentclass[oneside,a4paper,12pt]{article}

\usepackage{geometry} 
\usepackage{amsmath}
\usepackage{amsthm}
\usepackage{amsfonts,amssymb}
\usepackage{fancyhdr} 
\usepackage{mathrsfs}  
\usepackage{graphicx}
\usepackage[all]{xy}

\usepackage{times}
\usepackage{cases}
\usepackage{color}
\usepackage[colorlinks,linkcolor=blue,anchorcolor=blue,citecolor=blue]{hyperref}

\usepackage{float} 
\usepackage{subfig}
\usepackage{amsthm}

\newtheorem{theorem}{Theorem}[section]
\newtheorem{lemma}[theorem]{Lemma}

\newtheorem{definition}[theorem]{Definition}
\newtheorem{corollary}[theorem]{Corollary}

\newtheorem{claim}[theorem]{Claim}

\theoremstyle{remark}
\newtheorem{remark}[theorem]{Remark}

\theoremstyle{remark}

\numberwithin{equation}{section}

\begin{document}

\title{\textbf{Angle structures on pseudo 3-manifolds}}

\author{\medskip Huabin Ge, Longsong Jia, Faze Zhang}

\date{}

\maketitle
\begin{abstract}
It is still not known whether a hyperbolic $3$-manifold admits an angle structure or not. We consider angle structures with area-curvature $(A,k)$ on triangulated pseudo 3-manifolds $M$ in this article. A sufficient and necessary condition for the existence of such angle structures is established. As a consequence, any compact hyperbolic $3$-manifold with totally geodesic boundary admits an angle structure. We also derive certain topological information of $M$ from the existence of such angle structures.
\end{abstract}
\maketitle




\section{Introduction}
In the 1990s, Casson and Rivin (see \cite{Lackenby-1,Rivin} or \cite{Futer2011}  for a self-contained exposition) discovered a powerful technique for solving Thurston's hyperbolic gluing equations. By introducing the concept of angle structures on an ideally triangulated $3$-manifold $M$, they proved that if a maximal volume angle structure exists, it provides solutions to Thurston's gluing equations, thereby yielding a hyperbolic structure on $M$. Using theories and tools such as normal surfaces, the connection between angle structures and three-dimensional topology is increasingly being discovered. We refer to \cite{HRS,KangRubin,KR,Lac,Rivin} for examples. For a triangulated pseudo $3$-manifold $(M,\mathcal T)$, Luo-Tillmann \cite{LT} further developed the concept of (generalized) angle structures with area-curvature $(A, \kappa)$.
Assuming $A\leq 0$, a necessary condition for the existence of such angle structures was established in \cite{LT}.

In this article, we show that the necessary condition in \cite{LT} is also sufficient for the existence of such angle structures, by assuming that $(M,\mathcal T)$ admits a semi-angle structure. Let $\{\sigma_1,\cdots,\sigma_n\}$ be the tetrahedra in a triangulated compact pseudo $3$-manifold $(M,\mathcal{T})$. A \emph{normal surface} $F$ in $(M,\mathcal T)$ is an embedded surface, so as its intersection with any tetrahedron $\sigma_i$ is a disjoint union of normal disks. There are 7 types of normal disks in each tetrahedron $\sigma_i$, including 3 quadrilaterals and 4 triangles. The \emph{normal coordinate} of $F$ is an ordered normal disk types $(q_1,\cdots, q_{3n}, t_1,\cdots, t_{4n})\in\mathbb{R}^{7n}$. These coordinates satisfy a system of compatibility equations with solution space denoted by $\mathcal C(M, \mathcal T)$.

By taking the contributions of each normal disk $t_i$ or $q_j$ to the Euler characteristic of $F$, a function $\chi^\ast(t_j)$ or $\chi^\ast(q_j)$ can be introduced. Linearly extending those $\chi^\ast(t_i)$ and $\chi^\ast(q_j)$, the \emph{generalized Euler characteristic function} $\chi^\ast(\cdot)$ is well-defined on $\mathbb{R}^{7n}$ (see (\ref{e4})-(\ref{e6})). It agrees with the Euler characteristic on embedded and immersed normal surfaces, and gives an upper bound for the Euler characteristic of a branched normal surface. Each generalized angle structure $\alpha$ determines an area-curvature function $(A,k)$ as follows. The \emph{combinatorial area} $A(d)$ of a normal disk $d\subset \sigma_i$ is $\sum_j\alpha(e_j)-(k-2)\pi$, where $d$ is a $k$-side polygon, and intersects the edge $e_1,\cdots, e_k$ of $\sigma_i$. In addition, for an interior (boundary, resp.) edge $e$, its \emph{combinatorial curvature} $k(e)$ is $2\pi$ ($\pi$, resp.) minus the sum of the angles surrounding it. Similar to $\chi^\ast$, let $\chi^{(A, \kappa)}(\cdot)$ (see (\ref{equation3}) for details) be the \emph{Euler characteristic function with area-curvature $(A, \kappa)$}, our first result is:

\begin{theorem}
\label{theorem1}
Let $(M, \mathcal T)$ be a triangulated compact pseudo $3$-manifold. Suppose it admits a semi-angle structure with area-curvature $(A, \kappa)$ satisfying $A \leq 0$. If $\chi^\ast(s)<\chi^{(A, \kappa)}(s)$ for those $s\in C(M,\mathcal T)$ with all quadrilateral coordinates non-negative and at least one quadrilateral coordinates positive, then $(M, \mathcal T)$ admits an angle structure with area-curvature $(A, \kappa)$.
\end{theorem}

Combining Theorem \ref{theorem1} with Proposition $17$ of Luo-Tillmann \cite{LT}, we actually provide the following necessary and sufficient conditions for the existence of angle structures.

\begin{corollary}
\label{corollary2}
Let $(M, \mathcal T)$ be a triangulated compact pseudo $3$-manifold. Suppose it admits a semi-angle structure with area-curvature $(A, \kappa)$ satisfying $A \leq 0$. Then the following statements are equivalent:
\begin{enumerate}
\item  $(M, \mathcal T)$ admits an angle structure with area-curvature $(A, \kappa)$;
\item $\chi^\ast(s)<\chi^{(A, \kappa)}(s)$ for those $s\in C(M,\mathcal T)$ with all quadrilateral coordinates non-negative and at least one quadrilateral coordinates positive.
\end{enumerate}
\end{corollary}

Given an area-curvature $(A,k)$ with $A\leq 0$, if $A=0$ occurs precisely on the $(0,0,\pi)$-normal triangles and $A<0$ on the other normal triangles, we refer to such a pair $(A,\kappa)$ a \emph{flat area-curvature function}. This leads us to:

\begin{theorem}\label{theorem3}
Let $(M, \mathcal T)$ be a triangulated compact pseudo $3$-manifold, and $(A,\kappa)$ be a flat area-curvature function. If $(M, \mathcal T)$ admits a semi-angle structure with area-curvature $(A, \kappa)$ and at least one dihedral angle around an edge of $\mathcal T$  lies within $(0,\pi)$, then there exists an angle structure with area-curvature $(A', \kappa)$ where $A'<0$.
 \end{theorem}

As an application of Theorem \ref{theorem3}, we can establish the following:

\begin{corollary}\label{corollary4}
Suppose $M$ is a compact hyperbolic $3$-manifold with totally geodesic boundary. There is an ideal triangulation $\mathcal T$ so that $(M, \mathcal T)$ admits an angle structure.
 \end{corollary}

Conversely, the existence of an angle structure with area-curvature $(A, \kappa)$ on $(M, \mathcal T)$ reflects certain topological properties of $M$. Inspired by the work of Jaco-Rubinstein \cite{JR}, Lackenby \cite{Lackenby-1,Lac2} and Garoufalidis-Hodgson-Rubinstein-Segerman \cite{Garouf}, we have:

\begin{theorem}\label{theorem5}
For $A\leq0$ and $\kappa\leq0$, let $M$ be a compact orientable pseudo  $3$-manifold with non-$S^2$ boundary, and let $\mathcal T$ be an ideal triangulation of the interior of $M$. Then the existence of an angle structure with area-curvature $(A, \kappa)$ on $(M, \mathcal T)$ implies that $M$ contains no essential spheres, essential tori, essential disks, or essential annuli. Moreover, the ideal triangulation $\mathcal T$ is $0$-efficient.
\end{theorem}

It must be mentioned that, there had been several closely related concepts derived from angle structures, e.g. the semi-angle structure \cite{KangRubin,KR}, the taut angle structure \cite{Lac} and the generalized angle structure with area-curvature $(A,k)$ \cite{LT}. The angle structure with area-curvature $(0, 0)$ is actually the angle structure in the usual meaning. Luo-Tillmann \cite{LT} showed that for $3$-manifolds with torus or Klein bottle boundaries, the statements $1$ and $2$ of Corollary \ref{corollary2} are equivalent. Following that, Hodgson-Rubinstein-Segerman \cite{HRS} proved that each cusped hyperbolic $3$-manifold with some certain topological conditions has an ideal triangulation that admits an angle structure. As an application, following the spirit of Lackenby \cite{Lac2} and Hodgson-Rubinstein-Segerman \cite{HRS}, our Corollary \ref{corollary2} can be used to proved that, under quite weak topological assumptions, every cusped hyperbolic $3$-manifolds with totally geodesic boundary has an ideal triangulation that admits an angle structure \cite{GJZ}.

The paper is organized as follows. In Section \ref{2}, we introduces fundamental concepts, including polyhedral decomposition theory, ideal triangulation, angle structure, normal surfaces, Farkas's lemma, and Kojima's decomposition. In Section \ref{3}, we employ Farkas's lemma along with the arguments presented by Luo-Tillmann to prove Theorem \ref{theorem1} and Corollary \ref{corollary2}. In Section \ref{4}, we utilize a specific combinatorial method to establish Theorem \ref{theorem3} and Corollary \ref{corollary4}. Finally, Section \ref{5} is dedicated to the proof of Theorem \ref{theorem5}.

~

\noindent
\textbf{Acknowledgements:}
The authors are very grateful to Professor Ruifeng Qiu, Feng Luo, Tian Yang for many discussions on related problems in this paper. The first two authors would like to thank Professor Gang Tian for long-term support and encouragement. Huabin Ge and Longsong Jia are supported by NSFC, no.12341102, Faze Zhang is supported by NSFC, no.12471065.

\section{Preliminaries}\label{2}

\subsection{Triangulations}\label{subsection:2.1}
Following Jaco-Rubinstein~\cite{JR} and Luo~\cite{Luo-3-flow}, we have the following contents.

Let $X=\{\sigma_{i},i=1,2,...,n\}$ be a collection of pairwise disjoint $3$-simplices, and let $\Phi$ be a set of affine isomorphisms between the faces of $3$-simplices in $X$. If $\phi \in \Phi$, then $\phi$ maps the point $x$ on a face of $\sigma_{i}$ to a point $\phi(x)$ on a face of $\sigma_{j}$, where it is possible that $i=j$. The isomorphisms of $\Phi$ thus define an equivalence relation between the points of $X$. The resulting quotient space $X/\Phi$ is a pseudo $3$-manifold. We denote the quotient mapping from $X$ to $X/\Phi$ by $p$. It is noted that the non-manifold points in $X/\Phi$ only exist in the images under $p$ of the centers of some edges and the vertices of $\sigma_{i}$ for $i=1,2,...,n$.

\begin{definition}
Let $\mathcal T$ be the collection of images under $p$ of all faces of $\sigma_{i}$ for $i=1,2,...,n$ in $X$ and denote  $M=X/\Phi$. Then $\mathcal T$ is called a \textbf{triangulation} of $M$. Usually, we also denote that $M=|\mathcal T|$. Under the quotient mapping $p$, all images of $3$-simplices in $X$ are also denoted by $\{\sigma_{i},i=1,2,...,n\}$ and
the images of $0$-faces, $1$-faces, $2$-faces and $3$-faces of $\sigma_{i}(i=1,2,...,n)$ are called \textbf{ vertices, edges, faces and tetrahedra} of $\mathcal T$ respectively.
\end{definition}

We also use $T^{(i)}$ for $i=0,1,2,3$ to denote the the  vertices, edges, faces and tetrahedra of $\mathcal T$ respectively. By Moise~\cite{Moise} and Bing~\cite{Bing}, every compact $3$-manifold has a triangulation.

\begin{definition}\label{def1}
If the link of each vertex in $|\mathcal T|$ is a closed non-sphere surface, then $\mathcal T$ is called an \textbf{ideal triangulation} of $M=|\mathcal T|\backslash|\mathcal T^{(0)}|$. Where the vertices are called the \textbf{ideal vertices}.
\end{definition}

It is noted that with an ideal triangulation $\mathcal T$, $M=|\mathcal T|\setminus|\mathcal T^{(0)}|$ is the interior of a compact pseudo $3$-manifold with non-sphere boundaries. The following are the precisely description.

Suppose $M$ is a compact pseudo $3$-manifold with non-empty boundary whose genera are larger than or equal to $1$. Let $C(M)$ be the compact pseudo $3$-manifold obtained by coning off each boundary component of $M$ to a point. In
particular, if  $\partial M$ has $k$ components, then $C(M)$ has exactly $k$ cone points,
denoted by $\{v_1,...,v_k\}$. Hence $C(M)-\{v_1,...,v_k\}\cong M\setminus\partial M$. The ideal triangulation $\mathcal T$ of $M\setminus\partial M$ is actually the ideal triangulation of $C(M)\setminus\{v_1,...,v_k\}$ as defined in Defintion~\ref{def1}. Observed that the link of each $v_i$ for $i=1,...,k$ in $\mathcal T$, up to isotopy, is exactly the corresponding boundary component of $M$.   By Moise~\cite{Moise}, the interior of every compact $3$-manifold has an ideal triangulation.

\begin{definition}\label{def2}
Let $M$ be a compact pseudo $3$-manifold with non-sphere boundaries and $\mathcal T$ be an ideal triangulation of the interior of $M$. Let $N_i$ be an open regular neighborhood of the ideal vertex $v_i$ in $\mathcal T$ for $i=1,...,k$. If $\widetilde{\mathcal T}=\mathcal T\setminus \sqcup_{i=1}^{k} N_i$ is homeomorphic to  $M$, then  $\widetilde{\mathcal T}$ is called an \textbf{ideal truncated  triangulation} of $M$. Usually, the triangle boundaries of each $3$-cell in $\widetilde{\mathcal T}$ are called the \textbf{truncated vertices}.
\end{definition}

\subsection{Normal surfaces and the compatibility equations}
\label{subsection:2.2}

Let $\mathcal T$ be a triangulation of a pseudo $3$-manifold $M=X/\Phi$ and $\sigma$ be a tetrahedron in $\mathcal T$. A properly embedded arc in a $2$-face $f$ of $\sigma$ such that its end points lie in the interior of two distinct edges on $f$ is called a \emph{normal arc} of $f$. A \emph{normal disk} $D$ in $\sigma$ is a properly embedded disk in  $\sigma$ such that the intersection of $D$ and each $2$-face of $\sigma$ is either a normal arc or empty set. Notice that in $\sigma$, up to isotopy, there are at most seven types of normal disks which consist of four normal triangles and three normal quadrilaterals.

\begin{definition}
A compact surface $F$ in $(M,\mathcal T)$ is called \textbf{normal} if the intersections of $F$ and each tetrahedron of $\mathcal T$ are either empty or a disjoint union of normal disks.
\end{definition}

Let $\{\sigma_1,\cdots,\sigma_n\}$ be all the tetrahedra of $\mathcal T$. Since there are at most seven types of normal disks in each tetrahedron of $\mathcal T$, we can fix an ordering of all normal disc types $(q_1,\cdots, q_{3n}, t_1,\cdots, t_{4n})$ in $\mathcal T$, where $q_i$ denotes a normal quadrilateral type and $t_j$ a normal triangle type for $1\leq i\leq 3n, 1\leq j\leq 4n$. For a given normal surface $F$ in $(M,\mathcal T)$, its normal coordinate $\overline{F}=(x_1,\cdots,x_{3n},y_{1},\cdots,y_{4n})$ is a vector in $\mathbb{R}^{7n}$, where $x_i$ is the number of normal discs of type $q_i$ in $F$, and $y_j$ is the number of normal discs of type $t_j$ in $F$. A properly embedded normal surface $F$ is uniquely determined up to normal isotopy by its normal coordinate, see \cite{LT}\cite{Tillmann}.

Let $f$ be a face shared by two tetrahedra of $(M,\mathcal T)$, say $\sigma_{k}$ and $\sigma_{l}$. Notice that $f$ is not in $\partial M$. Let $a$ be a normal isotopy class of a normal arc in $f$. There are two types of normal disks in $\sigma_{k}$, one is type $x_{i}$ for normal quadrilaterals, the other is type $y_{j}$ for normal triangles, such that when restricted to $f$, the corresponding normal arcs are of the same type as $a$. Similarly, there are also two types of normal disks in $\sigma_{l}$, one is type $x_{i'}$ for normal quadrilaterals, the other is type $y_{j'}$ for normal triangles with similar properties. Thus, the normal coordinate $\overline{F}$ satisfies a linear equation at the face $f$ as follows
\begin{equation}\label{equation1}
x_{i}+y_{j}=x_{i'}+y_{j'}.
\end{equation}
All the equations at every such face of $\mathcal T$ are called the \emph{compatibility equations}. We denote $\mathcal C(M,\mathcal T)$ by the set of all vectors $(x_1,\cdots,x_{3n},y_1,\cdots,y_{4n})\in\mathbb{R}^{7n}$ satisfying (\ref{equation1}) at any such face.

For $\mathcal C(M,\mathcal T)$, Kang-Rubinstein~\cite{KR} introduced a basis
$$\big\{W_{\sigma_{i}},\,W_{e_{j}}\,|\,i=1,\cdots,n,\,j=1,\cdots,m\big\},$$
where $n$ is the number of tetrahedra in $\mathcal T$ and $m$ is the number of edges (not in $\partial M$) in $\mathcal T$. Then for any $s\in \mathcal C(M,\mathcal T)$, there exists a unique coordinate $(\omega_{1},\cdots,\omega_{n},z_{1},\cdots,z_{m} )$ such that
\begin{equation}\label{equation2}
s=\sum_{i=1}^{n}\omega_{i}W_{\sigma_{i}}+\sum_{j=1}^{m}z_{j}W_{e_{j}}.
\end{equation}

The above definitions can be deduced into the other triangulation cases.
\begin{remark}
Suppose $M$ is the interior of a compact $3$-manifold with non-empty boundary and $\mathcal T$ is an ideal triangulation of $M$. Since each edge of $\mathcal T$ is not in $\partial M$ and any face of ideal tetrahedra is attached to another face in $\mathcal T$, the \textbf{normal arc},  \textbf{normal disk}, \textbf{normal surface} and \textbf{compatibility equations} in $(M,\mathcal T)$ are defined the same as the above.
\end{remark}

Since the definition of ideal truncated  triangulations is based on the ideal triangulations, we also have the following interpretations.

\begin{remark}\label{rem}
Suppose $M$ is a compact pseudo $3$-manifold with non-empty boundary  and $\widetilde{\mathcal T}$ be an ideal truncated triangulation of $M$. If we call the edges and faces not in $\partial M$(in $\partial M$) as \textbf{internal edges}(\textbf{external edges}) and \textbf{internal faces}(\textbf{external edges}) respectively, then the \textbf{normal arc} is defined as the properly embedded arc in a internal face $f$ of $\widetilde{\mathcal T}$ such that its end points lie in the interior of two distinct internal edges on $f$.  The following definitions as \textbf{normal disk}, \textbf{normal surface} and \textbf{compatibility equations} in $(M,\widetilde{\mathcal T})$ are the same as the case of ideal triangulations.
\end{remark}

\subsection{Angle structures and Angle structures with area-curvature $(A, \kappa)$}\label{subsection:2.3}
The theory of angle structures has been extensively studied by other scholars, one can refer to \cite{KR}, \cite{Lac2}, \cite{Luo-3-flow}, \cite{Luo-Y} and so on.

\begin{definition}
\label{angle-structure}
Let $M$ be a compact pseudo $3$-manifold with non-empty boundary and  $\mathcal T$ be an ideal triangulation on the interior of $M$ with $3$-simplices $\sigma_1, ... ,\sigma_n$. An \textbf{angle structure} on $(M,\mathcal T)$ is a function $\alpha$ that, for each edge $e_{ij}$ of $\sigma_i$ ($1\leq i\leq n$, $1\leq j \leq 6$), assigns a values $\alpha(e_{ij})\in(0,\pi)$ which is called the dihedral angle, so that:
\begin{enumerate}
\item for any edge $e$ in $\mathcal T$, the sum of all dihedral angles around $e$ is $2\pi$;
\item if $v$ is an ideal vertex of some $\sigma_i$ with $1\leq i\leq n$ such that $v$ is corresponds with the tori boundary, then the sum of all dihedral angles adjacent to $v$ in $\sigma_{i}$ is $\pi$;
\item if $v$ is an ideal vertex of some $\sigma_i$ with $1\leq i\leq n$ such that $v$ is corresponds with the negative Euler characteristic surfaces boundary, then the sum of all dihedral angles adjacent to $v$ in $\sigma_{i}$ is less than $\pi$.
\end{enumerate}
\end{definition}
In the above definition, if all the dihedral angles are allowed to be taken in the closed interval $[0,\pi]$, it is said to be a \emph{semi-angle structure} on $(M,\mathcal T)$. By the work of Bao-Bonahon \cite{BB} ( or see \cite{Luo-Y} and \cite{Schlenker}), any angle structure on $(M,\mathcal T)$ endows each single $3$-simplex $\sigma_i$ with a hyperbolic geometry, making it an ideal tetrahedron or a hyperideal tetrahedron according to its type.

\begin{definition}
\label{angle-structure2}
Let $M$ be a compact $3$-manifold with non-empty boundary and  $\mathcal T$ be an ideal truncated triangulation on  $M$ with $3$-simplices $\sigma_1, ... ,\sigma_n$. An \textbf{angle structure} on $(M,\mathcal T)$ is a function $\alpha$ that, for each external edge $e$ of $\sigma_i$ ($1\leq i\leq n$), assigns the value $\frac{\pi}{2}$ and  for each internal edge $e_{ij}$ of $\sigma_i$ ($1\leq i\leq n$, $1\leq j \leq 6$), assigns a value $\alpha(e_{ij})\in(0,\pi)$ where $\frac{\pi}{2}$ and $\alpha(e_{ij})$ are called the \textbf{dihedral angles}, so that:
\begin{enumerate}
\item for any external edge $e$ in $\mathcal T$, the sum of all dihedral angles around $e$ is $\pi$;
\item for any internal edge $e$ in $\mathcal T$, the sum of all dihedral angles around $e$ is $2\pi$;
\item if the truncated vertex $v$ is in $\sigma_i$ with $1\leq i\leq n$ such that $v$ is corresponds with the tori boundary, then the sum of all dihedral angles at internal edges adjacent to $v$ in $\sigma_{i}$ is $\pi$;
\item if the truncated vertex $v$ is in $\sigma_i$ with $1\leq i\leq n$ such that $v$ is corresponds with the negative Euler characteristic surfaces boundary, then the sum of all dihedral angles at internal edges adjacent to $v$ in $\sigma_{i}$ is less than $\pi$.
\end{enumerate}
\end{definition}
By Kojima~\cite{Kojima}, any angle structure on $(M,\mathcal T)$ endows each single $3$-simplex $\sigma_i$ with a hyperbolic geometry, making it a truncated hyperideal tetrahedron.

The following definitions come from Luo-Tillmann~\cite{LT}.

As defined in Section \ref{subsection:2.1}, let $M=X/\Phi$ be a pseudo $3$-manifold with  a triangulation $\mathcal T$ and $\sigma$ be a tetrahedron of $\mathcal T$.  Suppose $D$ is a normal disk with vertices $\{v_1,\cdots,v_p\}$ for $p=3,4$  in $\sigma$. A \emph{combinatorial angle structure} on $D$ is a function $a:\{v_1,\cdots,v_p\}\rightarrow \mathbb R$ and the value $a(v_i)$ is called the angle at $v_i$. The \emph{combinatorial area} of $D$ given by a combinatorial angle structure $a$ on $D$, denoted by $A(D)$, is defined as
\begin{equation}
\label{e1}
A(D)=\sum_{i=1}^{k}a(v_i)-(k-2)\pi.
\end{equation}

Let $e$ be an edge in the triangulation $\mathcal T$ of $M$  and $\sigma_1,\cdots,\sigma_{n_1}$ be the tetrahedra of $\mathcal T$ that share $e$ as a common edge. Suppose there is a function $b_j:(e,\sigma_j)\rightarrow \mathbb R$ for $1\leq j \leq n_1$, then the \emph{curvature} of $e$, denoted by $\kappa(e)$, is defined as

when $e$ is not in $\partial M$,
\begin{equation}
\label{e2}
\kappa(e)=2\pi-\sum_{j=1}^{n_1}b_j(e,\sigma_j).
\end{equation}

and when $e$ is contained in $\partial M$,
\begin{equation}
\label{e3}
\kappa(e)=\pi-\sum_{j=1}^{n_1}b_j(e,\sigma_j).
\end{equation}

\begin{definition}
\label{angle-structure}
Let $\mathcal T$ be a triangulation of a psendo 3-manifold $M$ with 3-simplices $\{\sigma_1, \cdots ,\sigma_n\}$ and $(A,\kappa)$  be a pair of functions that acts on all normal triangles and all edges in $\mathcal T$. A \textbf{generalised angle structure with area-curvature $(A,\kappa)$} on $(M,\mathcal T)$ is a function $\alpha$ that, for each edge $e_{ik}$ of $\sigma_i$ ($1\leq i\leq n$, $1\leq k \leq 6$), assigns a real values $\alpha(e_{ik})$ which is called the \textbf{dihedral angle}, so that:
\begin{enumerate}
\item For any vertex $v$ of some $\sigma_i$ with $1\leq i\leq n$ in $\mathcal T$,  there is a normal triangle $t$ of $\sigma_i$ which is correspond to $v$. Then the combined area of $t$ given by the function $\alpha$  is $A(t)$.

\item for any edge $e$ in $\mathcal T$, the curvature of $e$ given by the function $\alpha$ is $\kappa(e)$;
\end{enumerate}
\end{definition}

In the above definition, if all the dihedral angles are allowed to be taken in the closed interval $[0,\pi]$ or the open interval $(0,\pi)$ , it is said to be a \emph{semi-angle structure with area-curvature $(A,\kappa)$} or  a \emph{angle structure with area-curvature $(A,\kappa)$} respectively on $(M,\mathcal T)$.

The angle structure on ideal triangulations and ideal truncated triangulations are actually the special cases of the angle structure with area-curvature $(A,\kappa)$. In fact,

\begin{remark}
Suppose $M$ is the interior of a compact $3$-manifold with tori or Klein bottles boundaries and  $\mathcal T$ is an ideal triangulation of $M$. Then the angle structure of $(M,\mathcal T)$ is actually the angle structure  with area-curvature $(0,0)$ on $(M,\mathcal T)$.
\end{remark}


\begin{remark}
Suppose $M$ is a compact $3$-manifold with non-empty boundary and $\widetilde{\mathcal T}$ is an ideal truncated triangulation of $M$. If the boundary surfaces of $M$ are consist of tori(or Klein bottles) and negative Euler characteristic surfaces, then the angle structure on $(M,\widetilde{\mathcal T})$ is the angle structure with area-curvature $(A\leq 0,\kappa=0)$ on $(M,\widetilde{\mathcal T})$. If the boundary surfaces of $M$ are consist of negative Euler characteristic surfaces, then the angle structure on $(M,\widetilde{\mathcal T})$ is the angle structure with area-curvature $(A< 0,\kappa=0)$ on $(M,\widetilde{\mathcal T})$.
\end{remark}

\subsection{Area-curvature function}\label{2.4}
Let $(M, \mathcal T)$ be a pseudo $3$-manifold with a triangulation $\mathcal T$ and $F$ be an embedded or immersed normal surface $F$ in $(M, \mathcal T)$. If the coordinate of $F$ is $(x_1,\cdots,x_{3n}, y_{1},\cdots,y_{4n})$ where $n$ is the number of tetrahedra in  $\mathcal T$. Then the \emph{generalized Euler characteristic} function $\chi^{*}: \mathbb{R}^{7n}\rightarrow \mathbb{R}$ is defined as
\begin{equation}
\label{e4}
\chi^{*}(x_1,...,x_{3n},y_{1},...,y_{4n} )=\sum_{i=1}^{3n}x_{i}\chi^{*}(q_{i})+\sum_{j=1}^{4n}y_{j}\chi^{*}(t_{j}),
\end{equation}
where $q_{i}$ and $t_{j}$ are the normal quadrilaterals and normal triangles in $F$ respectively.

In the above, for a normal triangle $t$ in a tetrahedra $\sigma_i$, if $b(t)$ is the number of normal arcs in $t\cap\partial M$ and $t\cap \sigma_i$ consists of three  edges $e_{ik},(k=1,2,3)$ with the edges valence $d_k$ where $1\leq k \leq 3$, then
\begin{equation}
\label{e5}
\chi^{*}(t)=-\frac{1}{2}\big(1+b(t)\big)+\sum_{k=1}^3\frac{1}{d_k}.
\end{equation}
For a normal quadrilateral  $q$ in a tetrahedra $\sigma_i$, if $b(q)$ is the number of normal arcs in $q\cap\partial M$ and $q\cap \sigma_i$ consists of four  edges $e_{ik'},(k'=1,2,3,4)$ with the edges valence $d_{k'}$ where $1\leq k' \leq 4$, then
\begin{equation}
\label{e6}
\chi^{*}(q)=-\frac{1}{2}\big(2+b(q)\big)+\sum_{k'=1}^4\frac{1}{d_{k'}}.
\end{equation}

If we denote $(x_1,\cdots,x_{3n}, y_{1},\cdots,y_{4n})$ by $\overline{F}$, then there is a fact that $\chi^{*}(\overline{F})=\chi(F)$. See Kang-Rubinstain~\cite{KR}.

If $\overline{F}$ corresponds with a element of $\mathcal C(M,\mathcal T)$, denoted by   $s=\sum\limits_{i=1}^{n}\omega_{i}W_{\sigma_{i}}+\sum\limits_{j=1}^{m}z_{j}W_{e_{j}}$, then the \emph{Area-curvature function} on $s$ is defined as
\begin{equation}\label{equation3}
\chi^{(A,\kappa)}(s)=\frac{1}{2\pi} (\sum_{t\in\bigtriangleup}y_{t}(s)A(t)+\sum_{j=1}^n 2z_j \kappa(e_j)),
\end{equation}
where $\bigtriangleup$ is the set of all the normal triangles of $s$, and $y_{t}$ is the corresponding normal triangle coordinate of $s$.

Luo-Tillmann (\cite{LT}, Lemma 15) got the following relationship between $\chi^{*}$ and $\chi^{(A,\kappa)}$:

\begin{lemma}\label{lemma2}
Suppose $(M, \mathcal T)$ admits an angle (or semi-angle) structure $\alpha$ with area-curvature $(A,\kappa)$. If $s\in \mathcal C(M,\mathcal T)$, then
\begin{equation}
\chi^{(A,\kappa)}(s)=\chi^{*}(s)-\frac{1}{2\pi}\sum_{q\in \square}A(q)x_{q}(s),
\end{equation}
where $\square$ is the set of all the normal quadrilaterals of $s$ and $A(q)$ is the combinatorial area of $q$ induced by the angle (or semi-angle) structure $\alpha$.

\end{lemma}


\subsection{Farkas's lemma}\label{sec1}

If $x=(x_{1},\cdots,x_{k})\in \mathbb{R}^{k}$, we use $x>0$ ($x\geq 0$, $x<0$, $x\leq 0$ resp.) to mean that all components of $x_{i}$ are positive (non-negative, negative, non-positive resp.). To approach our main results, we need the following duality result from linear programming, which is known as Farkas's lemma, and can be found, for instance, in \cite{Ziegler}. In the following lemma, vectors in $\mathbb{R}^{k}$ and $\mathbb{R}^{l}$ are considered to be column vectors.

\begin{lemma}\label{lemma}
Let $A$ be a real $k\times t$ matrix, $b\in \mathbb{R}^{k}$, and $\cdot$ be the inner product on $\mathbb{R}^{k}$.
\begin{enumerate}
\item $\{x\in \mathbb{R}^{t}|Ax=b\}\neq \emptyset$ if and only if for all $y\in\mathbb{R}^{k}$ such that $A^{T}y=0$, one has $y\cdot b=0$.
\item $\{x\in \mathbb{R}^{t}|Ax=b,x\geq 0\}\neq \emptyset$ if and only if for all $y\in\mathbb{R}^{k}$ such that $A^{T}y\leq 0$, one has $y\cdot b\leq0$.
\item $\{x\in \mathbb{R}^{t}|Ax=b,x>0\}\neq \emptyset$ if and only if for all $y\in\mathbb{R}^{k}$ such that $A^{T}y\neq0$ and $A^{T}y\leq0$, one has $y\cdot b<0$.
\end{enumerate}
\end{lemma}

\subsection{Kojima decomposition}\label{sec2}

To prove Theorem \ref{theorem5}, we need to demonstrate Kojima's decomposition of a hyperbolic $3$-manifold with totally geodesic boundary into truncated hyperideal polyhedra.

Following Bao-Bonahon~\cite{BB}, in the projective space model $B^{3}\subset RP^{3}$ of $\mathbb{H}^3$, a \emph{hyperideal polyhedra} refers to the intersection $\hat{P}$ of $\mathbb{H}^3$ with a compact convex polyhedra $\tilde{P}$ of $RP^{3}$ such that

\begin{enumerate}
\item each vertex of $\tilde{P}$ lies outside of $\mathbb{H}^3$ which is called the \emph{hyperideal vertex}, and
\item each edge of $\tilde{P}$ intersects $\mathbb{H}^3$ non-empty.
\end{enumerate}

\emph{Truncation} of a hyperideal polyhedra $\hat{P}$ at the hyperideal vertex $v$ in $\tilde{P}$ is defined as cutting off the thick end towards $v$ from $\hat{P}$ by a geodesic plane $H_v$ which is perpendicular to the adjacent geodesic faces at $v$. According to Kojima's lemma $2.1$~\cite{Kojima}, this truncation is unique. The resulting $3$-cell $P$ after truncation at every hyperideal vertex of $\tilde{P}$ is called a \emph{truncated  hyperideal polyhedron}. The boundary of $P$ consists of the intersections of $P$ with the faces of $\tilde{P}$ and the intersections of $\tilde{P}$ with the planes $H_v$. The former are referred as the \emph{internal faces} and the latter as the \emph{external faces} of $P$. The internal faces of $P$ are all right-angled hyperbolic polygons, while the external faces of $P$ are hyperbolic polygons.

It is noteworthy that a truncated hyperideal tetrahedron is a specific type of truncated hyperideal polyhedron where its untruncated polyhedron is a tetrahedron in $\mathbb{RP}^3$. See Figure  ~\ref{Fig2} for an example.

\begin{figure}[htbp]
\centering
\includegraphics[scale=1.15]{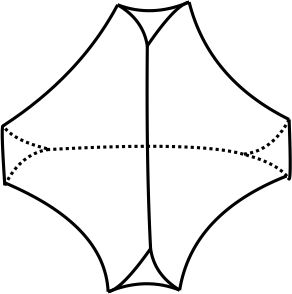}
\caption{Truncated Hyperideal Tetrahedron}
\label{Fig2}
\end{figure}

In~\cite{Kojima}, Kojima established the following theorem:

\begin{theorem}[Kojima]\label{theoremK}
Every hyperbolic $3$-manifold with totally geodesic boundary is a quotient of finitely many hyperideal polyhedra in $\mathbb{H}^3$ by isometries between pairs of faces.
\end{theorem}


\section{The proof of Theorem \ref{theorem1} and Corollary \ref{corollary2} }\label{3}
Let $M$ be a $3$-dimensional pseudo-manifold(with or without boundary) and $\mathcal T$ be an triangulation of $M$. We denote  all the tetrahedra in $\mathcal T$ by $\{\sigma_1,\cdots,\sigma_n\}$. Suppose $\alpha$ is a semi-angle structure(or angle structure) with area-curvature $(A,\kappa)$ where $A\leq 0$.

As in Definition \ref{angle-structure}, $\alpha$ assigns a dihedral angle to each edge $e_{ik}$ of $\sigma_i$ for $1\leq i \leq n$ and $1\leq k\leq 6$.  Then we can consider all the dihedral angles $\alpha_{ik}=\alpha(e_{ik})$  as  variables of the equations in Farkas' Lemma \ref{lemma}, see Figure \ref{Fig5}.

\begin{figure}[htbp]
\centering
\includegraphics[scale=0.35]{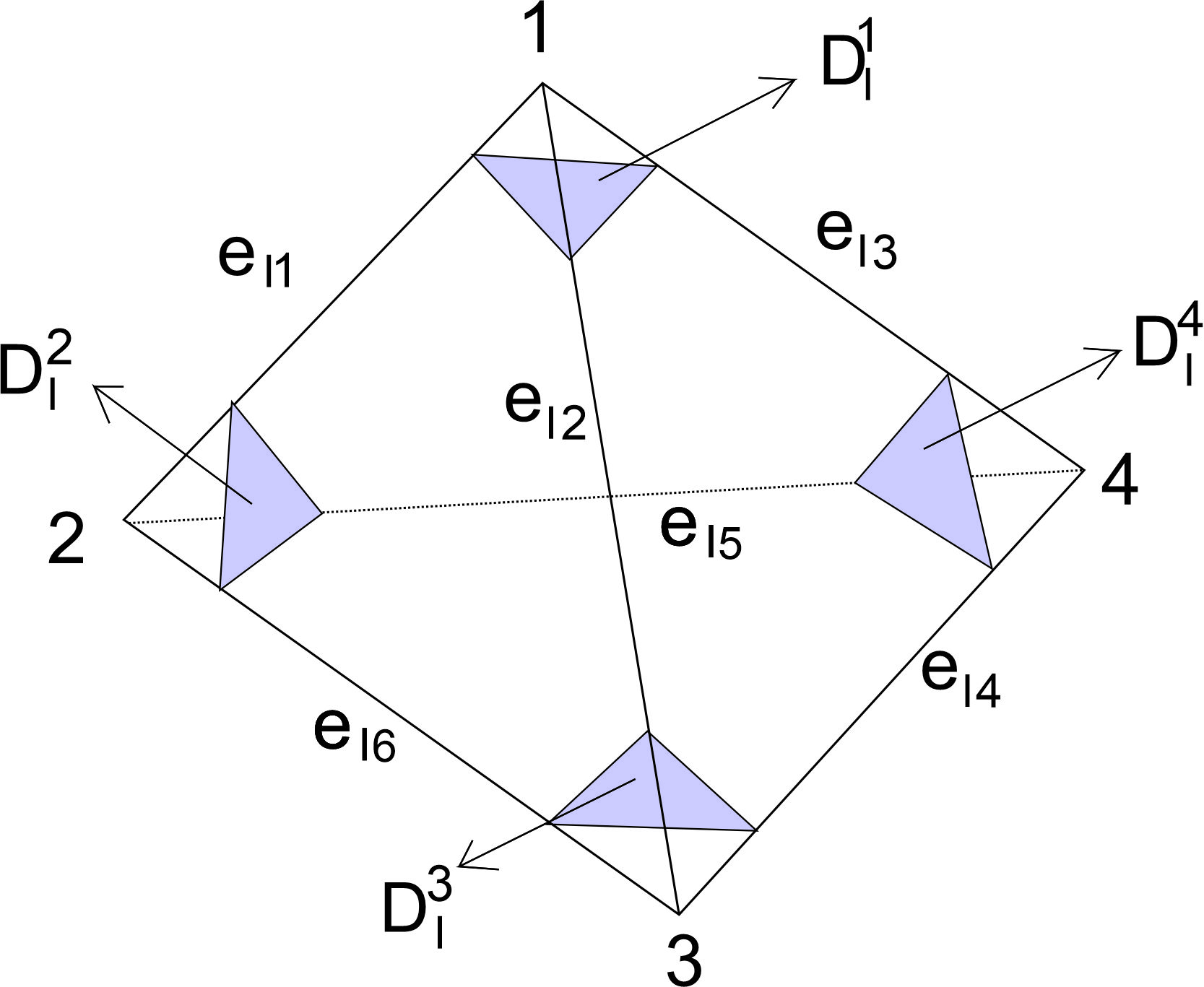}
\caption{the edge labelling and normal triangles in $\sigma_{i}$}
\label{Fig5}
\end{figure}

At the four normal triangles $D_{i}^1$, $\cdots$, $D_{i}^4$ of each tetrahedron $\sigma_{i}$, there are the following equations:

\begin{equation}\label{equ4}
  \begin{cases}
\alpha_{i1}+\alpha_{i2}+\alpha_{i3}=a_{i}^{1}\\
\alpha_{i1}+\alpha_{i5}+\alpha_{i6}=a_{i}^{2}\\
\alpha_{i3}+\alpha_{i4}+\alpha_{i5}=a_{i}^{3}\\
\alpha_{i2}+\alpha_{i4}+\alpha_{i6}=a_{i}^{4}.\\
  \end{cases}
\end{equation}
Notice that here  $a_{i}^{l}=A(t_{i}^{l})-\pi$ for $1\leq i \leq n$ and $1\leq l\leq 4$.

Suppose the number of edges in $\mathcal T$ is $m$. Then at each edge $e_j$ of $\mathcal T$ for $1\leq j\leq m$,  there are also the following equations:

\begin{equation}\label{equ5}
\sum_{i',k'}\alpha_{i'k'}^j=b_j.
\end{equation}
where $\alpha_{i'k'}^j=\alpha(e_{i'k'}^j)$ and $e_{i'k'}^j$ is actually the edge $e_{i'k'}$ in $\sigma_{i'}$ which is identified with the edge $e_j$ in $\mathcal T$.

Notice that here  $b_{j}=\pi-\kappa(e_i)$ or $2\pi-\kappa(e_i)$ for $1\leq j \leq m$ depending on $e_j$ contained in $\partial M$ or not.

Now consider the vector
 \begin{equation}
 (a,b)=(\cdots,a_{i}^{l},\cdots,b_j,\cdots)
 \end{equation}
where $a_{i}^{l}$ ($i=1,\cdots,n$, $k=1,\cdots,4$) comes from (\ref{equ4}) and $b_j$ comes from (\ref{equ5}).

Writing the system of equations (\ref{equ4}) and (\ref{equ5}) in a matrix form as
\begin{equation}\label{equation4}
Bx=(a, b)^{T}.
\end{equation}
Consider the transpose $B^T$ of $B$, which is also the dual of $B$. It has one variable $h_i^l$ for each normal triangle type $t_i^l$ and one variable $z_j$ for each internal edge $e_j$. If we denote the variables corresponding to $B^T $ by $(h, z)=(\cdots,h_{i}^{l},\cdots,\cdots,z_{j},\cdots)$, where $i=1,\cdots,n$, $l=1,\cdots,4$ and $j=1,\cdots,m$, then $B^T(h, z)^T=(\cdots,z_j + h_{i}^{l_1} + h_{i}^{l_2},\cdots)^{T}$.

The following useful formula (i.e., Formula (4.9) in \cite{LT}) belongs to Luo-Tillmann:

\begin{equation}\label{equation2}
\begin{aligned}
\frac{1}{\pi}(h, z)\cdot (a, b)&=\chi^{*}(s)-\chi^{(A,\kappa)}(s)\\
&\;\;\;\;+\frac{1}{2\pi}\sum (z_j + h_{i}^{l_1} + h_{i}^{l_2})(a_{i}^{l_1}+a_{i}^{l_2}-2\pi).
\end{aligned}
\end{equation}
where $s=\sum\limits_{i=1}^{n}\omega_i W_{\sigma_{i}}+\sum\limits_{j=1}^{m}z_j W_{e_{j}}\in \mathcal C(M, \mathcal T)$, and the summation runs over all internal edges in the whole triangulation $\mathcal{T}$.

Now we will prove Theorem \ref{theorem1} in the following:

\begin{proof}
Let $(M, \mathcal T)$ be the $3$-dimensional pseudo-manifold $M$, and let $(A, \kappa)$ be the area-curvature function defined on $\mathcal T$ with $A(t) \leq 0$ for each normal triangle $t$ as in Theorem \ref{theorem1}. Here $A\leq 0$ means that in Equation \ref{equ4}, each $a_{i}^{l}\leq \pi$ for $1\leq i\leq n$ and $1\leq l \leq 4$.

Let $\alpha$ be a semi-angle structure with area-curvature $(A, \kappa)$ which is the condition of Theorem \ref{theorem1}. Then $\alpha_{ik}=\alpha(e_{ik})\geq 0$ for $1\leq i\leq n$ and $1\leq k\leq 6$. Hence

\begin{equation}
(\alpha_{11},\cdots, \alpha_{n6})^T\in \big\{x\in \mathbb R^{6n}\mid Bx=(a,b),x\geq0\big\},
\end{equation}
which means that $\big\{x\in \mathbb R^{6n}\mid Bx=(a,b),x\geq0\big\}\neq \emptyset$. By the second part of Farkas's Lemma \ref{lemma}, for all $(h,z)\in \mathbb R^{4n+m}$ such that $B^T(h,z)^T\leq 0$, $(h,z)\cdot (a,b)\leq 0$ holds.

For getting an angle structure with area-curvature $(A, \kappa)$ from $\alpha$, by the third part of Farkas's Lemma \ref{lemma}, we only need to show that for all $(h,z)\in \mathbb R^{4n+m}$ such that $B^T(h,z)^T\neq 0$ and $B^T(h,z)^T\leq 0$, under the condition that $\chi^\ast(s)<\chi^{(A\leq 0, \kappa)}(s)$ for all $s\in C(M;\mathcal T)$ with all quadrilateral coordinates non-negative and at least one quadrilateral coordinates positive in Theorem \ref{theorem1}, $(h,z)\cdot (a,b)< 0$ holds.

Then there are two cases to be considered.

Case $1$, each coordinate of the vector $a=(a_1^1,\cdots, a_n^4)$ is equal to $\pi$ which means that $A=0$. Luo-Tillmann \cite{LT} has proved that $\chi^\ast(s)<\chi^{(A=0, \kappa)}(s)$ for all $s\in C(M;\mathcal T)$ with all quadrilateral coordinates non-negative and at least one quadrilateral coordinates positive implies that $(M, \mathcal T)$ satisfies an angle structure with area-curvature $(A=0, \kappa)$. Hence the proof is done.

Case $2$, there exists at least one coordinate of the vector $a=(a_1,\cdots, a_{n^4})$ is smaller than $\pi$. We rewrite the vector $a$ as $a=(\cdots, a_{i_1}^{l_1}<\pi,\cdots,\pi,\cdots,\pi)$. Then the corresponding area-curvature function is $(A_1<0,A_2=0, \kappa)$ and the corresponding matrix is denoted by $B'$ in Equation \ref{equation4}.

\begin{lemma}\label{claim}
There exists a vector $(\cdots,{a'}_{i_{1}}^{k_{1}},\cdots,\pi,\cdots,\pi,b_1,\cdots,b_m)=(a,b)$ with ${a'}_{i_{1}}^{k_{1}}<\pi$ such that
$$(h,z)\cdot (a,b)< 0$$
for all $(h, z)$ with ${B'}^{T}(h,z)^T\neq 0$ and ${B'}^{T}(h,z)^T\leq 0$.

\end{lemma}
\begin{proof}
Suppose each coordinate of the vector $(\cdots,{a'}_{i_{1}}^{k_{1}},\cdots,\pi,\cdots,\pi)$ is equal to $\pi$, then by Case $1$, the angle structure with area-curvature $(A=0, \kappa)$ exists. Hence by the third part of Farkas's Lemma \ref{lemma},
\begin{equation}\label{definition1}
 (h,z)\cdot (a_0,b)< 0,
\end{equation}
where $(a_0,b)=(\pi,\cdots,\pi,b_1,\cdots,b_m)$.

Now we can construct a function $F(t)$ defined as:  for $t\in [0,1]$ ,
\begin{equation*}
\begin{aligned}
F(t)(h,z)&=(h,z)\cdot (\cdots,{a'}_{i_{1}}^{l_{1}}+t(\pi-{a'}_{i_{1}}^{l_{1}}),\cdots,\pi,\cdots,\pi,b_1,\cdots,b_m)\\
&=\sum({a'}_{i_{1}}^{l_{1}}+t(\pi-{a'}_{i_{1}}^{l_{1}}))h_{i_{1}}^{l_{1}}+\sum \pi h_{i_{2}}^{l_{2}}+\sum_{l=1}^{m}b_j z_{j}.
\end{aligned}
\end{equation*}
Here $(h,z)$ are the variables corresponding to ${B'}^T $, and $h_{i_{2}}^{l_{2}}$ is the variable corresponding to $a_{i}^{l}=\pi$ in the equation (\ref{equ4}).

Thus we have $F(0)\leq 0$, and by Equation (\ref{equation1}), $F(1)< 0$ for all $B^{T}(h,z)^T\neq 0$ and $B^{T}(h,z)^T\leq 0$.

\begin{lemma}
$F(t)(h,z)<0$ for any $t\in (0,1)$, and for any variable $(h,z)$.
\end{lemma}
\begin{proof}
We prove it by contradiction. Assume there is a $t_0\in(0,1)$ and $(h_0,z_0)$ such that
\begin{equation*}
F(t_0)(h_0,z_0)\geq 0.
\end{equation*}
On the one hand $F(0)(h_0,z_0)\leq 0$, derived from the seme-angle structure on ($M, \mathcal T$), and on the other hand, $F(t)(h_0,z_0)$ as a function of $t$ is monotonic because of the linearity of $F(t)(h_0,z_0)$. Then $F(1)(h_0,z_0)\geq 0$, which contradicts the fact that $F(1)(h_0,z_0)<0$. Then we proved the above lemma.

\end{proof}
Finally, if set $a_{i_{1}}^{k_{1}}=\bar{a}_{i_{1}}^{k_{1}}+t_{0}(\pi-\bar{a}_{i_{1}}^{k_{1}})<\pi$, then we have
$$(h,z)\cdot (a,b)=F(t_{0})(h,z)< 0.$$
Hence we get the above Lemma \ref{claim}.
\end{proof}

By the third part of Farkas's Lemma \ref{lemma}, $$\big\{B'x=(a,b)^T,x> 0\big\}\neq \emptyset.$$
Hence $(M,\mathcal T)$ admits an angle structure with area-curvature $(A\leq 0, \kappa)$, and Theorem \ref{theorem1} is proved.

\end{proof}

For the proof of Corollary \ref{corollary2}, it is only need to prove that if $(M, \mathcal T)$ satisfies an angle structure with area-curvature $(A, \kappa)$, then $\chi^\ast(s)<\chi^{(A, \kappa)}(s)$ for all $s\in C(M;\mathcal T)$ with all quadrilateral coordinates non-negative and at least one quadrilateral coordinates positive. The following proof is belongs to Luo-Tillmann~\cite{LT}. However, for the integrity, we rewrite it here.

\begin{proof}
Since $(M, \mathcal T)$ satisfies an angle structure, say $\alpha$, with area-curvature $(A, \kappa)$, the equation $0< \alpha_{ik}=\alpha(e_{ik})< \pi$ holds for $1\leq i\leq n$ and $1\leq k\leq 6$. Hence

\begin{equation}
(\alpha_{11},\cdots, \alpha_{n6})^T\in \big\{x\in \mathbb R^{6n}\mid Bx=(a,b),x>0\big\},
\end{equation}
where the matrix $B$ is from Equation \ref{equation4} and the vector $(a,b)=(\cdots,a_i^l,\cdots,b_1,\cdots,b_m)$.

Then by the second part of Farkas's Lemma \ref{lemma}, $\big\{x\in \mathbb R^{6n}\mid Bx=(a,b),x>0\big\}\neq \emptyset$ implies that $(h,z)\cdot (a,b)<0$ for all $(h,z)\in \mathbb R^{4n+m}$ such that $B^T(h,z)^T\neq0$ and $B^T(h,z)^T\leq0$.

By Equation \ref{equation2}, for any $s\in \mathcal C(M, \mathcal T)$,

\begin{equation}\label{equation5}
\begin{aligned}
(h, z)\cdot (a, b)&=\pi (\chi^{*}(s)-\chi^{(A,\kappa)}(s))\\
&\;\;\;\;+\frac{1}{2}\sum (z_j + h_{i}^{l_1} + h_{i}^{l_2})(a_{i}^{l_1}+a_{i}^{l_2}-2\pi)<0.
\end{aligned}
\end{equation}

Also by $(h, z)\cdot (a, b)<0$, $B^T(h,z)^T<0$ which means that each coordinate $z_j + h_{i}^{l_1} + h_{i}^{l_2}$ of $B^T(h,z)^T$ is negative.  Since $A\leq 0$, each $a_{i}^{l}\leq \pi$ which implies that $a_{i}^{l_1}+a_{i}^{l_2}-2\pi\leq 0$. Hence

\begin{equation}
\begin{aligned}
(h, z)\cdot (a, b)&\geq \pi (\chi^{*}(s)-\chi^{(A,\kappa)}(s)).
\end{aligned}
\end{equation}

Combining with Equation \ref{equation5}, we get that $\chi^{*}(s)<\chi^{(A,\kappa)}(s)$.

Hence the proof of Corollary \ref{corollary2} is done.

\end{proof}

\section{The proof of Theorem \ref{theorem3} and Corollary \ref{corollary4}}\label{4}
Let $(M, \mathcal T)$ be a $3$-dimensional pseudo-manifold and $\{\sigma_1,\cdots,\sigma_n\}$ be all the tetrahedra in $\mathcal T$. Let $(A\leq 0, \kappa)$ be a flat area-curvature function on $\mathcal T$. Here flat means that if $D$ is a normal triangle such that $A(D)=0$, then its three angles  are $(0,0,\pi)$ and for the each  other normal triangle $D'$, $A(D')<0$. By the conditions of Theorem \ref{theorem3}, we suppose that $\alpha$ is a semi-angle structure with area-curvature $(A\leq 0, \kappa)$. Then for each edge $e_{ik}$ of each tetrahedron in $\mathcal T$, the following holds:

$$0\leq \alpha(e_{ik})=\alpha_{ik}\leq \pi,$$
where $1\leq i\leq n$ and $1\leq k\leq 6$.

Through slightly change the function $\alpha$, we can get our proof of Theorem \ref{theorem3} as follows:
\begin{proof}
For each edge $e_j$ of $\mathcal T$, let $m_1$, $n_1$ and $k_1$ respectively be the number of $0$-angles, $\pi$-angles and angles in $(0,\pi)$ around $e_j.$ By that at least one dihedral angles which around $e_j$ are belong to $(0,\pi)$, $k_1\neq 0$. Now let $t$ be a positive real number and $e_{i'k'}^j$ is the edge of $\sigma_{i'}$ which is identified with the edge $e_j$ in $\mathcal T$. Then

\begin{enumerate}
\item if both $m_1= 0$ and $n_1= 0$, then $\alpha_t(e_{i'k'}^j)=\alpha(e_{i'k'}^j)$;

\item otherwise,
\begin{enumerate}
\item when $\alpha(e_{i'k'}^j)=0,$  $\alpha_t(e_{i'k'}^j)=t$;
\item when $\alpha(e_{i'k'}^j)=\pi$,  $\alpha_t(e_{i'k'}^j)=\pi-3t$;
\item when $\alpha(e_{i'k'}^j)\in(0,\pi)$,  $\alpha_t(e_{i'k'}^j)=\alpha(e_{i'k'}^j)-\frac{m_1-3n_1}{k_1}t$.
\end{enumerate}
\end{enumerate}

Now let $A'$ be an area function on normal triangles $\{t_i^l\mid 1\leq i \leq n, 1\leq l\leq 4\}$. Then $A'(t_i^l)=(\alpha_t(e_{i1})+\alpha_t(e_{i2})+\alpha_t(e_{i3}))-\pi$ where $e_{i1},e_{i2},e_{i3}$ are the edges of $\sigma_i$ which intersect $t_i^l$ non-empty.
\begin{claim}
When the real number $t$ is sufficiently small, then the function $\alpha_t$ is an angle structure with area-curvature $(A', \kappa)$ where $A'<0$.
\end{claim}
\begin{proof}
It can be easily seen that in all cases as above, if $t$ is sufficiently small, then the dihedral angle $\alpha_t(e_{ik})\in (0,\pi)$ where $e_{ik}$ is the edge of $\sigma_i$ and $1\leq k \leq 6$.  And for each edge $e_j$ of $\mathcal T$,

\begin{equation}
\sum_{i',k'}\alpha_t(e_{i'k'}^j)=\sum_{i',k'}\alpha(e_{i'k'}^j)=b_j.
\end{equation}
where $e_{i'k'}^j$ is the edge of $\sigma_{i'}$ which is identified with the edge $e_j$ in $\mathcal T$. Hence for each edge $e_j$ in $\mathcal T$, $\kappa(e_j)=2\pi-b_j$.

For each normal triangle $t_i^l$ in $\sigma_i$ where $1\leq i \leq n$ and $1\leq l\leq 4$, without no loss of generality, denote the edges of $\sigma_i$ which intersect $t_i^l$ non-empty by $e_{i1},e_{i2}, e_{i3}$. Let $m_{1k}(k=1,2,3)$, $n_{1k}(k=1,2,3)$ and $k_{1k}(k=1,2,3)$ respectively be the number of $0$-angles, $\pi$-angles and angles in $(0,\pi)$ around $e_{ik}(k=1,2,3)$. For getting $A'<0$, there are two cases to be considered:
\begin{enumerate}
\item if the angles of $t_i^l$ defined by $\alpha$ are $(0,0,\pi)$, then
\begin{equation*}
A'(t_i^l)=t+t+(\pi-3t)-\pi=-t<0.
\end{equation*}
\item if all the angles of $t_i^l$ defined by $\alpha$ are belong to $(0,\pi)$ and suppose that $\frac{n_{13}}{k_{13}}\geq\frac{n_{11}}{k_{11}}$ and $\frac{n_{13}}{k_{13}}\geq\frac{n_{12}}{k_{13}}$, then

\begin{equation*}
\begin{aligned}
A'(t_i^l)&\leq \alpha(e_{i1})+\frac{3n_{11}}{k_{11}}t+\alpha(e_{i2})+\frac{3n_{12}}{k_{12}}t+\alpha(e_{i3})+\frac{3n_{13}}{k_{12}}t-\pi\\
&\leq \alpha(e_{i1})+\alpha(e_{i2})+\alpha(e_{i3})+9\frac{n_{13}}{k_{13}}t-\pi.
\end{aligned}
\end{equation*}

\end{enumerate}
Since $\alpha(e_{i1})+\alpha(e_{i2})+\alpha(e_{i3})<\pi$, if $t$ is sufficiently small, then $A'(t_i^l)<0$.
\end{proof}

Hence the proof of Theorem \ref{theorem3} is done.
\end{proof}

As an application of Theorem \ref{theorem3}, we show that a hyperbolic $3$-manifold with totally geodesic boundary has a truncated ideal triangulation $\widetilde{\mathcal T}$ such that $\widetilde{\mathcal T}$ admits an angle structure.

In the following, we will prove Corollary \ref{corollary4}.

\begin{proof}
Let $M$ be a hyperbolic 3-manifold with totally geodesic boundary. By Theorem \ref{theoremK}, there is a truncated hyperideal polyhedra decomposition $\mathcal P$ of $M$. We can firstly subdivide each truncated hyperideal polyhedron into truncated hyperideal tetrahedra and secondly insert flat truncated hyperideal tetrahedra between pairs of internal faces of the polyhedra. The resulting is an  ideal truncated triangulation of $M$ which is denoted by $\mathcal T$. The details are as follows which is combinatorially the same as that in Lackenby \cite{Lac2} and Hodgson-Rubinstein-Segerman\,\cite{HRS}.

All the polyhedra and polygons of $\mathcal P$ will respectively mean untruncated hyperideal polyhedra and untruncated hyperideal polygons. We use the projective model $\mathbb H^3\subset \mathbb {RP}^3.$ Recall that a polyhedron is a \emph{pyramid} if its faces consist of an $n$-gon and $n$ triangles which are the cone of the boundary of the $n$-gon to a point $v\in\mathbb{RP}^3.$ The point $v$ and the $n$-gon are respectively called the \emph{tip} and the \emph{base} of the pyramid.  Let $P\subseteq \mathcal P$ be an untruncated hyperideal polyhedron  and  $Q\subset\mathbb {RP}^3$ be the untruncated polyhedron of $\mathcal P$. We arbitrarily pick an vertex $v$ of $Q.$ Then there is a decomposition of $Q$  into pyramids with tips at $v$ and bases the faces of $Q$ disjoint from $v$ by taking cone at $v.$ For the base $D$ of each of the pyramids, we arbitrarily pick an vertex $w$ and decompose $D$ into  triangles by taking cone at $w.$ The decomposition of $D$ extends to a decomposition of the pyramid into tetrahedra. In this way, $Q$ is decomposed into a union of tetrahedra $\{\sigma_i\}.$ In turns, the intersections of $\{\sigma_i\}$ with $P$ give the decomposition of $P$ into truncated hyperideal tetrahedra. By the construction, each face of $P$ is decomposed into truncated hyperideal triangles. However, the decompositions of the same face from two different polyhedra adjacent to it may not always match which is called \emph{pillows}. See Figure  ~\ref{Fig3} for an example.

\begin{figure}[htbp]
\centering
\includegraphics[scale=0.75]{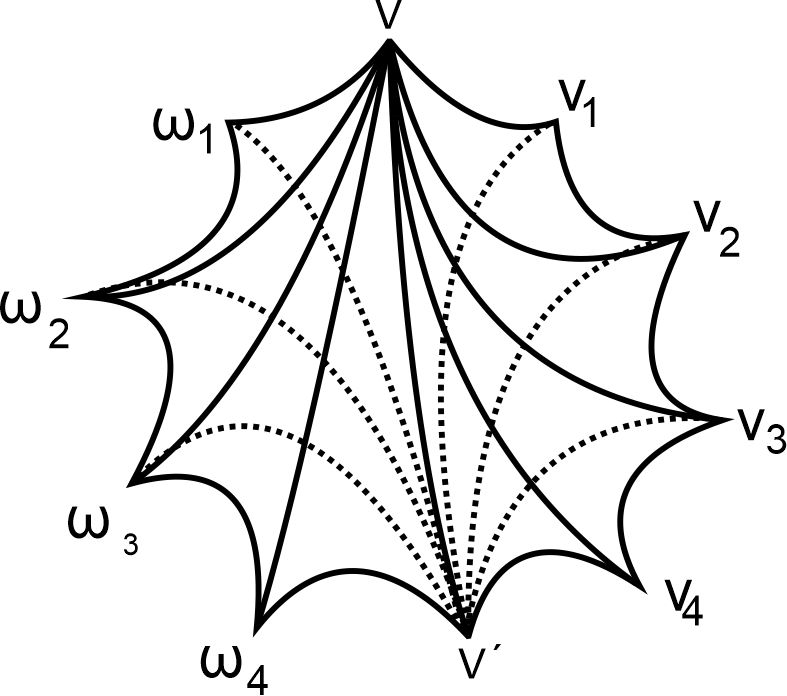}
\caption{Pillow}
\label{Fig3}
\end{figure}

For all the pillows, we insert enough flat truncated hyperideal tetrahedra into them. Then the resulting is an ideal truncated triangulation  $\mathcal T$ of $M$.

\begin{claim}\label{claim1}
For a pair $(A\leq 0,\kappa=0)$, there is a flat semi-angle structure $\alpha$ with area-curvature $(A, \kappa)$ such that $\mathcal T$ admits $\alpha$.
\end{claim}
\begin{proof}
For each internal edge $e_{ik}$ in $\sigma_i$ of $\mathcal T$, there is a function $\alpha$ such that $\alpha(e_{ik})$ is the dihedral angle at $e_{ik}$ in $\sigma_i$. By the construction of $\mathcal T$, $\alpha(e_{ik})<\pi$ if $\sigma_i$ is a hyperideal tetrahedron and $\alpha(e_{ik})=0$ or $\pi$ if $\sigma_i$ is flat. For each normal triangle $D_i^l$, there is a function $A$ such that
\begin{equation}
A(D_i^l)=\alpha(e_{i})+\alpha(e_{i2})+\alpha(e_{i3})-\pi,
 \end{equation}
where $e_{i1},e_{i2},e_{i3}$ are the internal edges of $\sigma_i$ which intersect $D_i^l$ non-empty.

If $D_i^l$ is a normal triangle in a flat truncated tetrahedron, then $A(D_i^l)=0$. Notice that in this case, the angles of $D_i^l$ assigned by $\alpha$ are $(0,0,\pi)$.

If $D_i^l$ is a normal triangle in a truncated tetrahedron(not flat), then $A(D_i^l)<0$ since $\partial M$ is a totally geodesic boundary in $M$.

For  each internal edge $e_{j}$ of $\mathcal T$, there is function $\kappa$ such that
\begin{equation}
\kappa(e_j)=2\pi-\sum_{i',k'}\alpha(e_{i'k'}^j),
 \end{equation}
where $e_{i'k'}^j$ is the internal edge of $\sigma_{i'}$ which is identified with the internal edge $e_j$ in $\mathcal T$. By that $M$ is a hyperbolic $3$-manifold, $\kappa(e_j)=0$.

Hence $\alpha$ is a flat semi-angle structure with area-curvature $(A, \kappa)$.
\end{proof}

Since no new edges were introduced in the construction of $\mathcal T$, each edge in $\mathcal T$ is adjacent to at least one hyperideal tetrahedron. Which means that for each internal edge $e_j$, there is at least one dihedral angles which around $e_j$ are belong to $(0,\pi)$. Hence by Theorem \ref{theorem3}, the proof of Corollary \ref{corollary4} is done.

\end{proof}

\section{The proof of Theorem \ref{theorem5}}\label{5}

Let $M$ be a compact pseudo  $3$-manifold with non-$S^2$ boundary and $\mathcal T$ be an ideal triangulation of the interior of $M$. Suppose $\alpha$ is an angle structure with area-curvature $(A\leq0, \kappa\leq0)$ on $(M, \mathcal T)$ and $S$ is an essential sphere, essential tori or Klein bottles. Then the Euler characteristic of $S$ is non-negative. By Haken \cite{Haken} and Freedman \cite{Freedman}, if $S$ is an embedded surface in $M$, then $S$ can be isotopied into an embedded surface in $(M, \mathcal T)$. $S$ is decomposed into some normal triangles or some normal quadrilaterals. Hence $S$ corresponds an element $s$ in $\mathcal C(M,\mathcal T)$ which is the solving space of compatibility equations. If we write $s$ as $s=\{x_1,\cdots,x_{3n},y_1,\cdots,y_{4n}\}$, then $x_i\geq 0$ and $y_{i'}\geq 0$ for $1\leq i\leq 3n, 1\leq i'\leq 4n$.

For each normal triangle $t$ and edge $e$ of $S$, by $(A\leq0, \kappa\leq0)$, we have $A(t)\leq0$ and $\kappa(e)\leq0$.

Recall that if $s=\sum_{i=1}^n\omega_iW_{\sigma_i}+\sum_{j=1}^m 2z_j W_{e_j}$, then

\begin{equation*}
\chi^{(A,\kappa)}(s)=\frac{1}{2\pi}( \sum_{t\in\bigtriangleup}y_{t}(s)A(t)+\sum_{j=1}^n 2z_j \kappa(e_j)),
\end{equation*}
where $2z_j$ is exactly the intersection number of $S$ with the edge $e_j$.

Hence $\chi^{(A,\kappa)}(s)\leq 0$.

By the following equation
\begin{equation*}
\chi^{(A,\kappa)}(s)=\chi^{*}(s)-\frac{1}{2\pi}\sum_{q\in \square}A(q)x_{q}(s),
\end{equation*}
and $\chi^{*}(s)=\chi(S)\geq0$, we get that

\begin{equation}\label{equ3}
\frac{1}{2\pi}\sum_{q\in \square}A(q)x_{q}(s)\geq 0.
\end{equation}

For any normal quadrilateral $q$, without loss of generality, let $e_{i2},e_{i3},e_{i5},e_{i6}$ be the four edges of $\sigma_i$ which intersect $q$ non-empty, see Figure~\ref{Fig3}.  Then

\begin{figure}[htbp]
\centering
\includegraphics[scale=0.30]{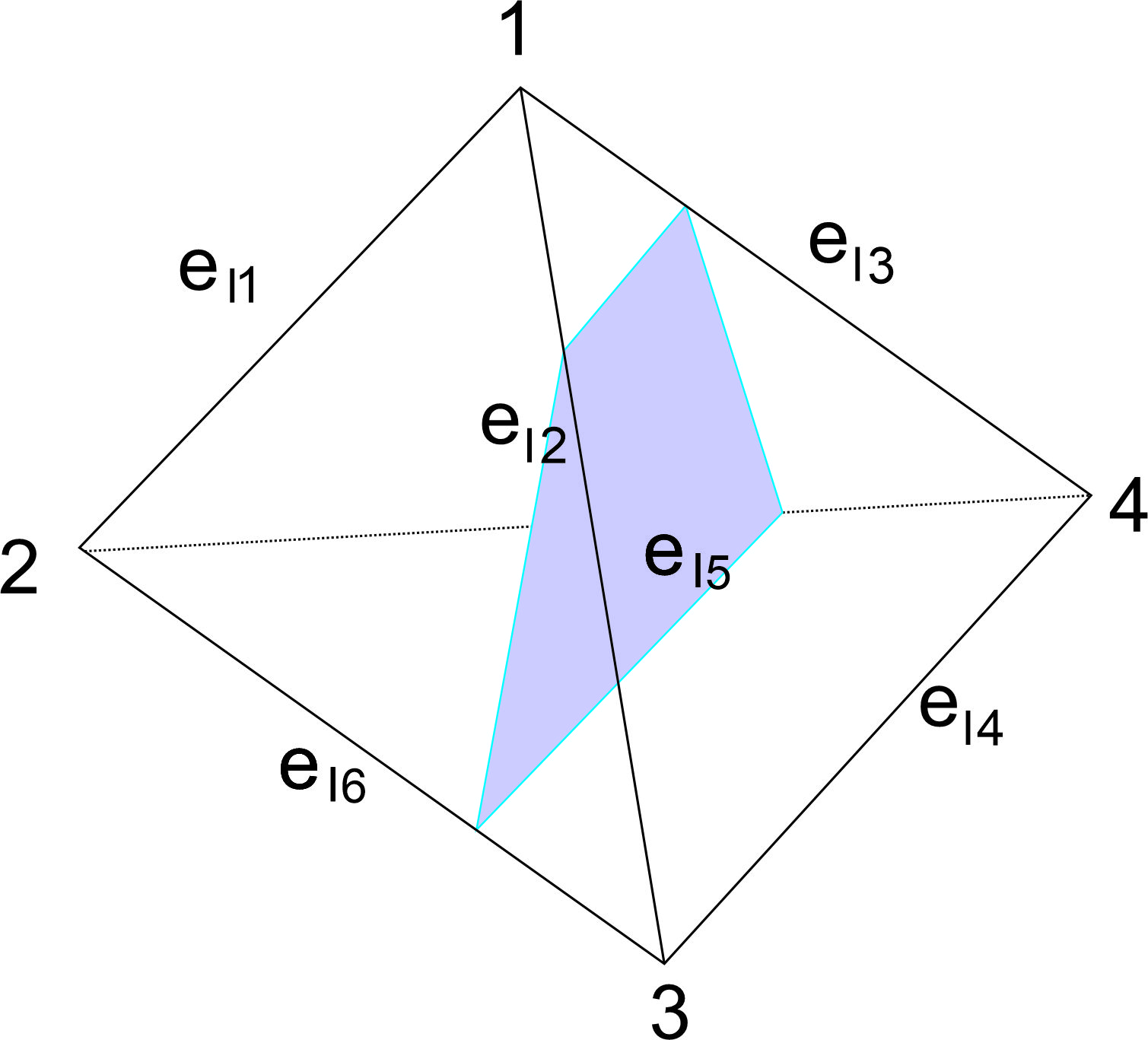}
\caption{Quadrilateral}
\label{Fig3}
\end{figure}

\begin{equation}
\begin{aligned}
A(q)&=\alpha(e_{i2})+\alpha(e_{i3})+\alpha(e_{i5})+\alpha(e_{i6})-2\pi \\
&<\alpha(e_{i2})+\alpha(e_{i3})+\alpha(e_{i1})+\alpha(e_{i5})+\alpha(e_{i6})+\alpha(e_{i1})-2\pi\\
&=(\alpha(e_{i2})+\alpha(e_{i3})+\alpha(e_{i1})-\pi)+(\alpha(e_{i5})+\alpha(e_{i6})+\alpha(e_{i1})-\pi)\\
&\leq 0.
\end{aligned}
\end{equation}
where the inequality holds because of that $\alpha(e_i^k)\in (0,\pi)$ for $1\leq i\leq n, 1\leq k\leq6$ and $A\leq 0$.

The above inequality combining with   $x_{q}(s)=x_i\geq 0$  is contradicts with the Equation \ref{equ3}. Hence $M$ contains no essential sphere, essential tori or Klein bottles.

Now suppose $\widetilde{\mathcal T}$ is an ideal truncated triangulation of $M$ which is deduced from $\mathcal T$ by truncating all the ideal vertices in $\mathcal T$. Then $\alpha$ is also an angle structure with area-curvature $(A\leq0, \kappa\leq0)$ on $(M, \mathcal T)$. Let $S$ be an essential disk or annulus in $M$. By Lackenby~\cite{Lackenby-1}, $S$ can be isotopied into an \emph{admissible surface} in $(M, \widetilde{\mathcal T})$. Which means that the intersections of $S$ and each truncated tetrahedron of $\widetilde{\mathcal T}$ are consistent of the following \emph{admissible disks} which intersect the external
faces of each truncated tetrahedron at most three times, see Figure  ~\ref{Fig4}.

\begin{figure}[htbp]
\centering
\includegraphics[scale=0.75]{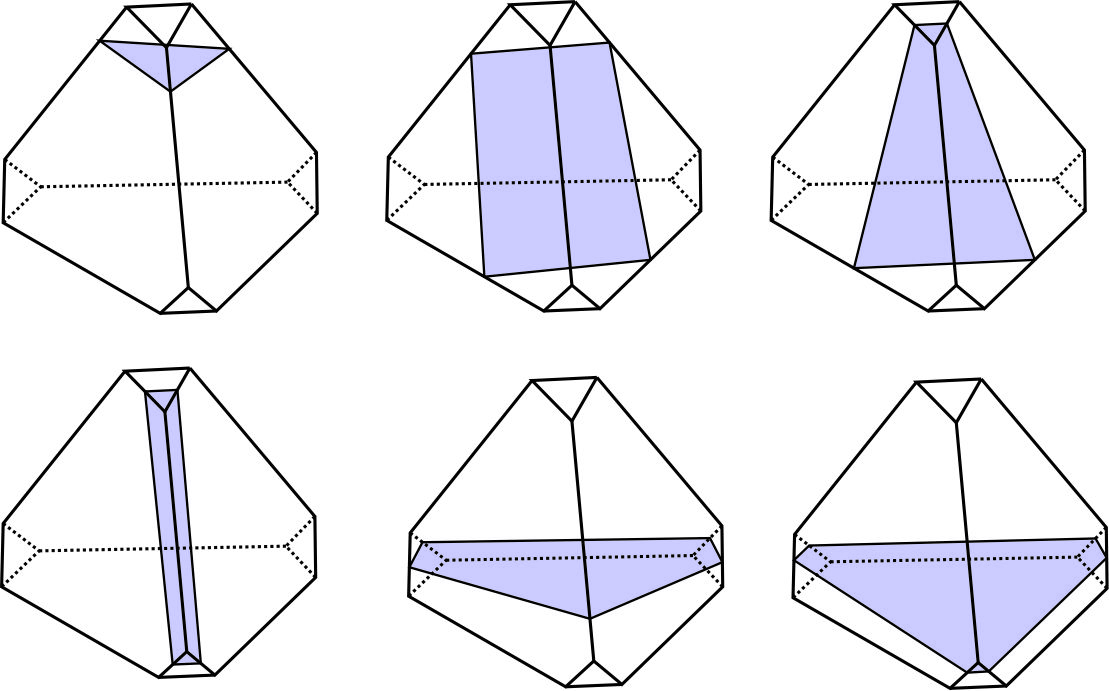}
\caption{Admissible surfaces}
\label{Fig4}
\end{figure}

By Lackenby~\cite{Lac2} and the same arguments as the case of ideal triangulation, we have that the Euler characteristic of $S$ is negative which is contradicts with the fact that $S$ is an essential disk or annulus. Hence,  $M$ contains no essential disk or annulus.

By Theorem $7.2$ of Jaco-Rubinstein~\cite{JR} which said that if  $M$ is a compact, irreducible, $\partial$-irreducible, anannular $3$-manifold, then any ideal triangulation of the interior of $M$ can be modified to a $0$-efficient ideal triangulation of the interior of $M$. This modification can be achieved by crushing $\mathcal T$ along certain isotopy copies of all the vertex-linking surfaces. Hence the proof of Theorem \ref{theorem5} is done.




\noindent
\noindent

\noindent Huabin Ge, hbge@ruc.edu.cn\\[2pt]
\emph{School of Mathematics, Renmin University of China, Beijing 100872, P. R. China}\\[2pt]

\noindent Longsong Jia, jialongsong@stu.pku.edu.cn\\[2pt]
\emph{School of Mathematical Sciences, Peking University, Beijing, 100871, P. R. China}\\[2pt]

\noindent Faze Zhang, zhangfz201@nenu.edu.cn\\[2pt]
\emph{School of Mathematics and Statistics, Northeast Normal University, Changchun, Jilin, 130024, P. R. China} \\[2pt]


\begin{thebibliography}{99}
\setlength{\itemsep}{-1pt} \small

\bibitem{BB} X. Bao, F. Bonahon, \emph{Hyperideal polyhedra in hyperbolic 3-space}, Bull. Soc. Math. France 130 (2002), no. 3, 457-491.



\bibitem{Bing} B. H. Bing, \emph{An alternative proof that $3$-manifolds can be triangulated}, Ann. Math. 69 (1959), 37-65.

\bibitem{Freedman} M. Freedman, J. Hass, P. Scott, \emph{Least area incompressible surfaces in $3$-manifolds}, Invent. Math. 71 (1983) 609-642.

\bibitem{Futer2011} D. Futer, F. Gu\'{e}ritaud, \emph{From angled triangulations to hyperbolic structures, Interactions between hyperbolic geometry, quantum topology and number theory}, Contemp. Math., vol. 541, Amer. Math. Soc., Providence, RI, 2011, pp. 159-182.
  
\bibitem{Garouf} S. Garoufalidis, C. D. Hodgson, J. H. Rubinstein, H. Segerman, \emph{$1$-efficient triangulations and the index of a cusped hyperbolic 3-manifold}, Geom. Topol. 19 (5): 2619-2689, 2015.

\bibitem{GJZ} H. Ge, L. Jia, F. Zhang, \emph{Angle structure on general hyperbolic $3$-manifolds}, preprint, arXiv:2408.14003 [math.GT], 2024.


 \bibitem{Haken} W. Haken, \emph{Some results on surfaces in 3-manifolds}, from: ``Studies in Modern Topology", Math. Assoc. Amer. (1968), 39-98.


\bibitem{HRS} C. D. Hodgson, J. H. Rubinstein, H. Segerman, \emph{Triangulations of hyperbolic 3-manifolds admitting strict angle structures}, J. Topol. 5  (2012),  no. 4, 887-908.

\bibitem{JR} W. Jaco, J. H. Rubinstein, \emph{$0$-Efficient triangulations of 3-manifolds}, J. Differential Geom. 65 (1) (2003), 61-168.

\bibitem{KangRubin}
E. Kang, J. H. Rubinstein,  \emph{Ideal triangulations of 3-manifolds I; spun normal surface theory}, Geometry and Topology Monographs, Vol. 7 , Proceedings of the Casson Fest,  (2004), 235-265.

\bibitem{KR} E. Kang, J. H. Rubinstein, \emph{Ideal triangulations of $3$-manifolds. II. Taut and angle structures}, Algebr. Geom. Topol.5 (2005), 1505-1533.


\bibitem{Lackenby-1} M. Lackenby, \emph{Word hyperbolic Dehn surgery}, Invent. Math. 140 (2000), 243-282.

\bibitem{Lac} M. Lackenby, \emph{Taut ideal triangulations of 3-manifolds}, Geom. Topol. 4 (2000), 369-395.

\bibitem{Lac2} M. Lackenby, \emph{An algorithm to determine the Heegaard genus of simple 3-manifolds with non-empty boundary}, Alge. Geom. Top. 8 (2008), 911-934.

\bibitem{Luo-3-flow} F. Luo, \emph{A combinatorial curvature flow for compact 3-manifolds with boundary}, Electron. Res. Announc. Amer. Math. Soc. 11 (2005), 12-20.


\bibitem{LT} F. Luo, S. Tillmann, {\em Angle structures and normal surfaces}, Trans. Amer. Math. Soc. 360 (2008), no. 6, 2849-2866.

\bibitem{Luo-Y} F. Luo, T. Yang, \emph{Volume and rigidity of hyperbolic polyhedral 3-manifolds}, J. Topol. 11 (2018), 1-29.

\bibitem{Kojima} S. Kojima, \emph{Polyhedral decomposition of hyperbolic 3-manifolds with totally geodesic boundary}, Aspects of low-dimensional manifolds, 93-112, Adv. Stud. Pure Math., 20 (1992), 93-112.


\bibitem{Moise} E. E. Moise, \emph{Affine structures in 3-Manifolds. V. The triangulation theorem and hauptvermutung}, Ann. Math. 56 (1952), 96-114.


\bibitem{Rivin} I. Rivin, {\em Euclidean structures on simplicial surfaces and hyperbolic volume}, Ann. of Math. 139 (1994), 553-580.

\bibitem{Schlenker} J. M. Schlenker, \emph{Hyperideal polyhedra in hyperbolic manifolds}, preprint (2002), arXiv math/0212355.

\bibitem{Tillmann} S. Tillmann, \emph{Normal surfaces in topologically finite 3-manifolds}, Enseign. Math. 54 (3-4) (2008), 329-380.

\bibitem{Ziegler} G. M. Ziegler, \emph{Lectures on polytopes}, Springer-Verlag, New York, 1995.\\[5pt]
\end{thebibliography}
\end{document}